\documentclass[hidelinks,11pt]{amsart}

\usepackage{amsmath,amssymb,amsfonts,hyperref,graphicx,tikz,mathrsfs,amsaddr,blindtext}

\usepackage[labelsep=period, font=footnotesize, labelfont=bf]{caption}

\makeatletter
\@namedef{subjclassname@2020}{2020 Mathematics Subject Classification}
\makeatother

\usepackage[margin=1in,footskip=0.25in]{geometry}

\theoremstyle{plain}

\newtheorem{theorem}{Theorem}[section]
\newtheorem{lemma}[theorem]{Lemma}

\theoremstyle{definition}
\theoremstyle{proposition}
\newtheorem{definition}[theorem]{Definition}

\newtheorem{corollary}[theorem]{Corollary}

\newtheorem{proposition}[theorem]{Proposition}

\numberwithin{equation}{section}

\def\R{\mathbb{R}}

\def \h{\mathcal{H}}

\def\w{\mathbf{w}}

\def\x{\mathbf{x}}

\def\y{\mathbf{y}}

\date{\today}

\begin{document}

\title[Subdivision method in the  Laplacian matching polynomial]
{Subdivision method in the  Laplacian matching polynomial}

\author[Jiang-Chao Wan, Yi Wang]{Jiang-Chao Wan, Yi Wang, Zhi-Yuan Wang}
\address{School of Mathematical Sciences, Anhui University, Hefei 230601, P. R. China}
\email{wanjc@stu.ahu.edu.cn, wangy@ahu.edu.cn, wangzhiyuan9641@126.com}
\thanks{{\it Corresponding author.} Yi Wang}
\thanks{{\it Funding.} Supported by the National Natural Science Foundation of China (No. 12171002)}
\date{\today}



\sloppy

\maketitle

\begin{abstract}
As a bridge connecting the matching polynomial and the Laplacian matching polynomial of graphs,
the subdivision method  is expected to be useful for investigating the Laplacian matching polynomial.
In this paper, we study applications of the method from three aspects.
 We prove that the zero sequence of the Laplacian matching polynomial of a graph majorizes its degree sequence, establishing a dual relation between the Laplacian matching polynomial and the characteristic polynomial of the signless Laplacian matrix of graphs.
 In addition, from different viewpoints, we give a new combinatorial interpretations for the coefficients of the Laplacian matching polynomial.

\end{abstract}


\section{Introduction}

All graphs considered in this article are assumed to be finite, undirected, and without
loops or multiple edges.
Let $G$ be a graph with the vertex set $V(G)$ and the edge set $E(G)$.
Let $M$ be a subset of $E(G)$.
We denote by $V(M)$ the set of vertices of $G$ each of which is an endpoint of an edge in $M$.
If no two distinct edges in $M$ share a common endpoint, then $M$ is called a {\it matching} of $G$.
The set of matchings of $G$ together with the empty set is denoted by $\mathcal{M}(G)$.
In 1972, Heilmann and Lieb \cite{Heilmann1,Heilmann} formally defined the {\it matching polynomial} of $G$ as
$$
\alpha(G,x) = \sum_{M \in \mathcal{M}(G)}(-1)^{|M|}x^{|V(G)\backslash V(M)|}
$$
 in studying statistical physics. It is an interesting object in algebraic combinatorics and has numerous applications in theoretical physics \cite{Heilmann1,Heilmann} and chemistry \cite{Aihara,Gutman}.
We refer the reader to textbooks \cite{Cvet,Godsil} for systemic introduction of the matching polynomial theory.

In  \cite{Heilmann}, Heilmann and Lieb proved that for a graph $G$ with maximum degree $\Delta(G) \geq 2$,
the zeros of $\alpha(G,x)$ lie in the interval $\big(-2\sqrt{\Delta(G)-1}, 2\sqrt{\Delta(G)-1}\big)$.
 Based on this, Marcus, Spielman, and Srivastava \cite{Marcus}  established that there are infinitely many bipartite
 Ramanujan graphs.
Another important result in the matching polynomial theory, summarized in \cite{Cvet, Godsil1},  provides a dual relation between the matching polynomial and the characteristic polynomial of graphs.
Before presenting it,  we need to introduce some needed notations.
Denote by $\omega(G)$ the number of components of $G$.
If $H$ is a subgraph of $G$, then we simply write $G-H$ for the induced subgraph of $G$ on $V(G)\setminus V(H)$.
The {\it adjacency matrix} of $G$, denoted by $A(G)$, is the matrix whose rows and columns are indexed by ${V}(G)$ with the $(u, v)$-entry is $1$ if $u$ and $v$ are adjacent in $G$ and $0$ otherwise.
We always use $\phi(M,x)$ to denote the characteristic polynomial of a square matrix $M$.

\begin{theorem}[\cite{Cvet}, Theorem 4.4]\label{matchingdualitythm}
Let $\mathcal{C}(G)$ be the set of all $2$-regular subgraphs of $G$. Then
\begin{equation}\label{matchingdualityeq1}
\phi(A(G), x)=\alpha(G,x)+\sum_{C \in \mathcal{C}(G)}(-2)^{\omega(C)}\alpha(G-C, x),
\end{equation}
and
\begin{equation}\label{matchingdualityeq2}
\alpha(G,x)=\phi(A(G), x)+\sum_{C \in \mathcal{C}(G)}2^{\omega(C)}\phi(A(G-C), x).
\end{equation}
\end{theorem}

Inspired by the above classical results
and the breakthrough work of Marcus, Spielman, and Srivastava \cite{Marcus,Marcus2},
Mohammadian \cite{Mohammadian} introduced a new real-rooted graph polynomial, called the {\it Laplacian matching polynomial}, which is defined as
$$
\beta(G, x)=\sum_{M\in\mathcal{M}(G)}(-1)^{|M|}\left(\prod_{v\in V(G)\setminus V(M)}\big(x-d_G(v)\big)\right),
$$
where $d_G(v)$ is the degree of $v$ in $G$.
Furthermore, Mohammadian \cite {Mohammadian} proved that $\beta(G,x)$ is the average of Laplacian characteristic polynomials of all signed graphs with underlying graph $G$, and all zeros of $\beta(G,x)$ lie in  the interval $\big[0, \Delta(G)+2\sqrt{\Delta(G)-1}\big)$ if ${\Delta}(G)\geq 2$.
Independently, Zhang and Chen \cite{Zhang} also studied this polynomial and called it {\it the average Laplacian polynomial} of $G$.
We refer the reader to papers \cite{LiYan,Mohammadian,Wan,Zhang} for more results on the Laplacian matching polynomial.

One of main methods dealing with the Laplacian matching polynomial is called the {\it subdivision method}, which establishes a connection between  the Laplacian matching polynomial of a graph and  the matching polynomial of its subdivision graph.
The method originated from the following identity found by Yan and Yeh  \cite{Yan}, 
\begin{equation}\label{subdivisiontheoremeqqq}
\alpha(S_G, x)=x^{|E(G)|-|V(G)|}\beta(G, x^2),
\end{equation}
where $S_G$ denotes the subdivision graph of $G$, that is, the graph obtained from $G$ by inserting into each edge of $G$ a vertex of degree $2$.

Recently, Wan, Wang and Mohammadian \cite{Wan} proposed a strengthening version of \eqref{subdivisiontheoremeqqq}, which has more potential applications in investigating the Laplacian matching ploynomial. For the sake of simplicity, we restate it by means of a new terminology, that is, the principal Laplacian matching polynomial with respect to a subgraph.
Let $H$ be an induced subgraph of $G$. The {\it principal Laplacian matching polynomial} of $G$ with respect to $H$ is defined to be the polynomial
\begin{equation}\label{inducedlmp}
\beta(G, x)_{[H]}=\sum_{M \in \mathcal{M}(H)}(-1)^{|M|}
\left(\prod_{v \in V(H)\backslash V(M)}\big(x-d_G(v)\big)\right).
\end{equation}
Note that $\beta(G, x)_{[G]}=\beta(G, x)$. Theorem 2.2 of \cite{Wan} can be presented as follows.
\begin{theorem}[\cite{Wan}]\label{gsubdivisiontheorem}
Let $G$ be a graph, and let $W$ be a subset of $V(G)$. Then
$$\alpha (S_G-W, x )= x^{|E(G)|-|V(G)|+|W|}\beta(G, x^2)_{[G-W]}.$$
\end{theorem}

As an application of Theorem \ref{gsubdivisiontheorem}, the authors of \cite{Wan} immediately obtained a sharp upper bound for the largest root of $\beta(G, x)$.
 Roughly speaking, the main idea of the subdivision method is to establish the connection between the Laplacian matching polynomial of graphs  and the matching polynomial of their subdivisions by applying Theorem \ref{gsubdivisiontheorem}.
 It is noticeable that, comparing to \eqref{subdivisiontheoremeqqq}, Theorem \ref{gsubdivisiontheorem} provides a successive vertex-deleted relation of the Laplacian matching polynomial, which is beneficial when we apply mathematical induction to prove some results. One may find this advantage in the proofs of Theorem \ref{interlace} and Theorem \ref{Qdualitythm2}.

%
%
%
%
%
%
%
%
%
%


The present paper is  devoted to exploring more applications of the subdivision method,  and is organized as follows.
In Section \ref{coefficientsection}, from different viewpoints, we give a new combinatorial interpretations for the coefficients of the Laplacian matching polynomial.
In Section \ref{majorizationsection}, we present a vertex-interlacing theorem for the principal Laplacian matching polynomial and further prove that the zero sequence of the Laplacian matching polynomial of a graph majorizes its degree sequence.
In Section \ref{dualitysection}, we establish a dual relation between the Laplacian matching polynomial and the characteristic ploynomial of the signless Laplacian matrix of graphs, which can be viewed as an analogue of Theorem \ref{matchingdualitythm}.
 The final section briefly discusses other potential applications of the subdivision method in the Laplacian matching polynomial of edge-weighted graphs.

\section{Combinatorial interpretation of the coefficients}\label{coefficientsection}

A basic result  due to Zhang and Chen \cite{Zhang} provides a combinatorial interpretation for the coefficients of the Laplacian matching polynomial by means of the weight of a TU-subgraph.
A {\em TU-subgraph} of a graph $G$ is a subgraph whose components are trees or unicyclic graphs.
Suppose that a TU-subgraph $H$ of $G$ consists of $c$ unicyclic graphs and trees $T_1,\ldots,T_t$.
Then, the {\em weight} of $H$ is defined as
$$\w(H)=2^c\prod_{i=1}^t|V(T_i)|.$$

\begin{theorem}[\cite{Zhang}]\label{zhangchen}
	Let $G$ be a graph on $n$ vertices and let $\beta(G, x ) = \sum_{r=0}^n(-1)^r a_rx^{n-r}.$
	Then
	$$a_r=\sum_{H\in \h_r(G)}\w(H),$$
	where $\h_r(G)$ denotes the set of all TU-subgraphs of $G$ with $r$ edges.
\end{theorem}

The purpose of this section is to give a new combinatorial interpretation for the coefficients of the Laplacian matching polynomial and a new proof of Theorem \ref{zhangchen} by applying the subdivision method. Recall that the {\it subdivision graph} of a  graph $G$, denote by $S_G$, is the graph obtained from $G$ by inserting into each edge $e =\{u, v\}$ of $G$ a
vertex of degree $2$, which is also denoted by $e$.
Thus, for an edge $e =\{u, v\}$ of $G$, it corresponds to two edges $\{e, u\}$ and $\{e, v\}$ of $S_G$.
Moreover, one may find that  $S_G$ is a bipartite graph with bipartition $\{V(G), E(G)\}$.

An $r$-matching in $G$ is a matching with $r$ edges.
We denote by $p(G,r)$ the number of $r$-matchings of $G$ with the convention that $p(G,0)=1$.
To obtain the main result of this section, we need to determine the numbers of the maximum matchings of the subdivision graphs of trees and unicyclic graphs.

\begin{lemma}\label{ccsubdivionlma1}
	Let $T$ be a tree with at least two vertices. Then the graph $S_T-v$ has exactly one perfect matching for any $v\in V(T)$.
\end{lemma}
\begin{proof} 
	By induction on $|V(T)|$, it is not hard to see that $S_T-v$ has at least one perfect matching.
	If there are two perfect matchings in $S_T-v$ , then we can find a cycle in $S_T-v$, which contradicts that $T$ is a tree.
\end{proof}

\begin{lemma}\label{comexpcoefsubdivisionglma}
	The following statements hold.
	\begin{enumerate}
		\item If $T$ is a tree, then $p(S_T, |E(T)|)=|V(T)|$. 
		\item If $U$ is a unicyclic graph, then $p(S_U, |E(U)|)=2$.
	\end{enumerate}
\end{lemma}
\begin{proof}
	Statement (1) is trivial for $|V(T)|=1$. Assume that $|V(T)|\geq 2.$
	Note that $S_T$ is a bipartite graph with bipartition $\{V(T), E(T)\}$, where $|V(T)| =|E(T)|+1$.  For any $v\in V(T)$, by Lemma \ref{ccsubdivionlma1}, $S_T-v$ has exactly one perfect matching, which consists of $|E(T)|$ edges. Statement (1) follows.

By the definiton of the subdivision graph,  $S_U$ is a bipartite graph with bipartition $\{V(U), E(U)\}$, where $|V(U)|=|E(U)|$.
If $U$ is a cycle, it is clear that $p(S_U, |E(U)|)=2$.
Assume that $U$ contain at least one vertex of degree $1$. 
Denote by $C$ the unique cycle in $U$, and denote by $T$ the tree obtained from $U$ by contracting $C$ to a new vertex $v$. Clearly, $S_C$ has exactly $2$ perfect matchings and $S_T-v$ has exactly one perfect matching by  Lemma \ref{ccsubdivionlma1}. 
Note that each perfect matching of $S_U$ is composed of a perfect matchings of $S_C$ and a perfect matchings of $S_T-v$. Therefore, $p(S_U, |E(U)|)=2$. This proves  statement (2).
\end{proof}

For each edge $f\in E(S_G)$,
denote by $e(f)$ the unique edge of $G$ corresponding to  $f$.
The following result asserts that every $r$-matching of $S_G$ uniquely corresponds to a TU-subgraph
 of $G$ with $r$ edges.
\begin{lemma}\label{swelldefinedlemma}
Let $G$ be a graph.
Then, for each $r$-matching $\{f_1,\ldots,f_r\}$ of $S_G$,
the edge-induced subgraph of $G$ induced by edge set $\{e(f_1),\ldots,e(f_r)\}$ is a TU-subgraph of $G$.
\end{lemma}
\begin{proof}
We prove the assertion by induction on $r$.
It is valid for $r=1$.
Suppose that $r \geq 2$ and the assertion holds for all positive integers less than $r$.
Let  $\{f_1,\ldots,f_r\}$ be an $r$-matching of $S_G$, and let $H$ be the  subgraph of $G$ induced by the edge set $\{e(f_1),\ldots,e(f_r)\}$.
Observe that $\{f_1,\ldots,f_r\}$ is an $r$-matching of $S_H$
that saturates the part $E(H)$, so we get that $|E(C)|\leq |V(C)|$ for every component $C$ of $H$.
By the induction hypothesis, the subgraph of $G$ induced by the edge set $\{e(f_1),\ldots,e(f_{r-1})\}$, say $\widehat{H}$,  is a TU-subgraph of $G$.
If $|e(f_r)\cap V(\widehat{H})|$ is equal to $0$ or $1$, then $H$ is also a TU-subgraph of $G$.
Assume that $|e(f_r)\cap V(\widehat{H})|=2$.
It is obvious that $H$ is a TU-subgraph of $G$ in the following three cases:
(1) $e(f_r)$ connects two tree components of $\widehat{H}$;
(2) $e(f_r)$ connects a tree component and a unicyclic component of $\widehat{H}$;
(3) $e(f_r)$ connects two vertices in a tree component of $\widehat{H}$.
It remains to consider that $e(f_r)$ connects two vertices in a unicyclic component of $\widehat{H}$ or it connects two unicyclic components of $\widehat{H}$.
In each case, we find that the new generated component, say $A$, has the property that $|E(A)|=|V(A)|+1>|V(A)|$, a contradiction.
\end{proof}

\begin{theorem}\label{subdivisionTU-subgraphs}
Let $G$ be a graph.
Then, for every natural number $r$, we have
\begin{equation}\label{rmequalTUidentity}
p(S_G,r)=\sum_{H\in \h_r(G)}\w(H),
\end{equation}
where $\h_r(G)$ denotes the set of all TU-subgraphs of $G$ with $r$ edges.
\end{theorem}
\begin{proof}
Let $\mathcal{M}_r(S_G)$ be the set of $r$-matchings of $S_G$.
Consider the map $\chi: \mathcal{M}_r(S_G) \rightarrow \h_r(G)$
which sends every $r$-matching $\{f_1,\ldots,f_r\}$ of $S_G$ to the  subgraph of $G$ induced by the edge set $\{e(f_1),\ldots,e(f_r)\}$.
Clearly, the map $\chi$ is well-defined by Lemma \ref{swelldefinedlemma}.
To prove \eqref{rmequalTUidentity}, it suffices to prove that, for each TU-subgraph $H\in \h_r(G)$,
the cardinality of the inverse image $\chi^{-1}(H)$ of $H$ is equal to $\w(H)$.
Applying the multiplication principle, we just need to consider the contribution of each component of $H$.
More precisely, for a component $Y$ of $H$ with $s$ edges, we need to prove that
the number of $s$-matchings of $S_Y$ is equal to $s+1$ if
$Y$ is a tree, and is equal to $2$ if $Y$ is a unicyclic graph.
This is provided by Lemma \ref{comexpcoefsubdivisionglma}. The result follows.
\end{proof}

Combining Theorem \ref{gsubdivisiontheorem} and Theorem \ref{subdivisionTU-subgraphs}, we immediately obtain a new combinatorial interpretation for the coefficients of the Laplacian matching polynomial.

\begin{theorem}\label{xomexpar}
Let $G$ be a graph on $n$ vertices and let $\beta(G, x ) = \sum_{r=0}^n(-1)^r a_rx^{n-r}.$
Then
$$a_r=p(S_G,r)=\sum_{H\in \h_r(G)}\w(H),$$
where $\h_r(G)$ denotes the set of all TU-subgraphs of $G$ with $r$ edges.
\end{theorem}
\begin{proof}
Let $m=|E(G)|$.
Applying Theorem \ref{gsubdivisiontheorem}, we have
	\begin{align*}
	\beta(G, x )&=x^{\frac{n-m}{2}}\alpha(S_G, x^{\frac{1}{2}})\\
	&=x^{\frac{n-m}{2}}\sum_{r\geq 0}(-1)^rp(S_G,r) x^{\frac{n+m-2r}{2}}   \\
	&=\sum_{r\geq 0}(-1)^r p(S_G,r) x^{n-r}.
	\end{align*}
This implies that $a_r=p(S_G,r)$. The assertion follows from this fact and Theorem \ref{subdivisionTU-subgraphs}.
\end{proof}

\section{Interlacing and majorization}\label{majorizationsection}

Let $\alpha_1\leq\cdots\leq\alpha_{n-1}$ and $\beta_1\leq\cdots\leq\beta_n$
be respectively the zeros of two real-rooted polynomials $f$ and $g$ with $\deg f=n-1$ and $\deg g=n$.
We say that the zeros of $f$ {\it interlace} the zeros of $g$ if
$$\beta_1\leq\alpha_1\leq\beta_2\leq \cdots \leq  \alpha_{n-1}\leq\beta_n.$$
Heilmann and Lieb \cite{Heilmann} proved that, for each vertex $v$ of a graph $G$,
the zeros of $\alpha(G-v,x)$ interlace the zeros of $\alpha(G,x)$.
Moreover, \cite[Corollary 6.1.3]{Godsil} states that if $G$ is connected, then  the largest zero of $\alpha(G,x)$ has the multiplicity $1$ and is greater than the largest zero of $\alpha(G-v,x)$.
Combining these facts and Theorem \ref{gsubdivisiontheorem}, we obtain the following vertex-interlacing theorem for the principal Laplacian matching polynomial.
We would like to remark that an edge-interlacing theorem for the Laplacian matching polynomials is already shown in \cite[Theorem 3.10]{Wan}.
For a real-rooted polynomial $f(x)$, in what follows, we respectively write $\lambda_{\max}(f(x))$ and $\lambda_{\min}(f(x))$ for the largest and the least root of $f(x)$.

\begin{theorem}\label{interlace}
Let $H$ be an induced subgraph of a graph $G$.
Then the zeros of $\beta(G,x)_{[H-v]}$ interlace the zeros of $\beta(G,x)_{[H]}$ for each vertex $v$ of $H$.
In particular, if $H$ is connected, then the largest zero of $\beta(G,x)_{[H]}$ has the multiplicity $1$ and is greater than the largest zero of $\beta(G,x)_{[H-v]}$.
\end{theorem}
\begin{proof}
Since $H$ is an induced subgraph of $G$, there is a subset $W$ of $V(G)$ such that $H=G-W$.
Given  a vertex $v\in V(H)$, and let $h=|V(H)|$. Note that Theorem \ref{gsubdivisiontheorem} implies that all roots of any principal Laplacian matching polynomial of $G$ are real and nonnegative. 
Let $\alpha_1\leq\cdots\leq\alpha_{h-1}$ and $\beta_1\leq\cdots\leq\beta_{h}$ be  the zeros  of  $\beta(G,x)_{[H-v]}$ and $\beta(G,x)_{[H]}$, respectively.
By Theorem \ref{gsubdivisiontheorem},  the zeros of $\alpha(S_G-W-v, x)$ and  $\alpha(S_G-W, x)$ can be arranged as 
	$$-\sqrt{\alpha_{h-1}} \leq \cdots \leq -\sqrt{\alpha_1}\leq \cdots \leq \sqrt{\alpha_1}\leq\cdots\leq\sqrt{\alpha_{h-1}},$$ and
$$-\sqrt{\beta_{h}} \leq  \cdots \leq -\sqrt{\beta_1} \leq \cdots \leq \sqrt{\beta_1}\leq\cdots\leq\sqrt{\beta_{h}},$$ respectively.
 By \cite[Corollary 6.1.3]{Godsil}, the zeros of $\alpha(S_G-W-v, x)$ interlace the zeros of $\alpha(S_G-W, x)$, then we have $$ \sqrt{\beta_1}\leq\sqrt{\alpha_1}\leq\sqrt{\beta_2}\leq\cdots\leq\sqrt{\alpha_{h-1}}\leq\sqrt{\beta_h},$$
which implies that
	$ \beta_1\leq\alpha_1\leq\beta_2\leq\cdots\leq\alpha_{h-1}\leq\beta_h$, as desired.

We next prove the `in particular' statement.
It is trivial for the case $|V(H)|=1$.
Assume that $H$ is connected with $|V(H)|\geq 2$.
Note that $S_G-W$ consists of some possible isolated vertices and a nontrivial component containing $S_H$, say $\widehat{H}$.
Then, we have 
\begin{equation}\label{equation3.11111}
\lambda_{\max}(\alpha(S_G-W,x))=\lambda_{\max}(\alpha(\widehat{H} ,x))
\end{equation}
with the same multiplicity.
Moreover, \cite[Corollary 6.1.3]{Godsil} states that the largest zero of $\alpha(\widehat{H},x)$ has the multiplicity $1$ and
\begin{equation}\label{equation3.11122}
\lambda_{\max}(\alpha(\widehat{H},x))> \lambda_{\max}(\alpha(\widehat{H}-v,x)).
\end{equation}
Applying Theorem \ref{gsubdivisiontheorem}, we conclude that the largest zero of $\beta(G,x)_{[H]}$ has the multiplicity $1$.
Since $S_G-W-v$ is the union of $\widehat{H}-v$ and some possible isolated vertices,
  we have
\begin{equation}\label{equation3.11133}
\lambda_{\max}(\alpha(S_G-W-v,x))=\lambda_{\max}(\alpha(\widehat{H}-v,x)).
\end{equation}
Combining \eqref{equation3.11111}--\eqref{equation3.11133} and Theorem \ref{gsubdivisiontheorem}, we deduce that
\begin{align*}
\lambda_{\max}(\beta(G,x)_{[H]})
=&\big(\lambda_{\max}(\alpha(S_G-W,x))\big)^2 && \text {(by Theorem \ref{gsubdivisiontheorem})} \\
=&\big(\lambda_{\max}(\alpha(\widehat{H},x))\big)^2 && \text {(by \eqref{equation3.11111})}\\
>&\big(\lambda_{\max}(\alpha(\widehat{H}-v,x))\big)^2 && \text {(by \eqref{equation3.11122})}  \\
=&\big(\lambda_{\max}(\alpha(S_G-W-v,x))\big)^2  && \text {(by \eqref{equation3.11133})}  \\
=&\lambda_{\max}(\beta(G,x)_{[H-v]}). && \text {(by Theorem \ref{gsubdivisiontheorem})}
\end{align*}
This completes the proof.
\end{proof}

\begin{corollary}\label{maxmimbound}
Let $G$ be a connected graph with minimum degree $\delta$. Then
 $\lambda_{\min}(\beta(G, x))\leq  \delta $ with equality holds if and only if $|V(G)|=1$.
\end{corollary}
\begin{proof}
The upper bound is trivial if $|V(G)|=1$.
Assume that $|V(G)| \geq 2$.
Consider an edge $uv \in E(G)$ with $d_G(u)=\delta$, and let $H$ be the subgraph of $G$ induced by $\{u,v\}$.
Then
$$
\beta(G, x)_{[H]} =(x-\delta)(x-d_G(v))-1=x^2-(2\delta+a)x+\delta(\delta+a)-1,
$$
where $a=d_G(v)-\delta \geq 0$.
By the quadratic formula, we obtain that
\begin{eqnarray*}
\lambda_{\min}( \beta(G, x)_{[H]})
&=&\frac{1}{2}\left(2\delta+a -\sqrt{  ( 2\delta+a )^2-4\big(\delta(\delta+a)-1\big) }\right)\\
&=&\frac{1}{2}\left(2\delta+ a - \sqrt{ a^2+4 }\right)\\
&< &\delta.
\end{eqnarray*}
The result follows from this fact and Theorem \ref{interlace}.
\end{proof}


It is known that the sequence of the Laplacian eigenvalues of a graph
majorizes its degree sequence, see \cite{Brouwer} or \cite{Grone}.
Inspired by this classical result in spectral graph theory,
we next prove that the zero sequence of the Laplacian matching polynomial of a graph majorizes its degree sequence.
Let us begin with the definition of majorization.
We rearrange the components of $\x =(\x_1, \ldots, \x_n)\in \R^n$ in nonincreasing
order as $\x_{[1]}\geq \cdots \geq \x_{[n]}$.

\begin{definition}
Let $\x =(\x_1, \ldots, \x_n),$ $\y =(\y_1, \ldots, \y_n) \in \R^n$. If for every $k=1, \ldots,n,$
$$\sum_{i=1}^k \x_{[i]} \leq \sum_{i=1}^k\y_{[i]},$$ 
then we say that $\y$ weakly majorizes $\x$.
If $\y$ weakly majorizes $\x$ and $\sum_{i=1}^n \x_i=\sum_{i=1}^n \y_i$, then we say that $\y$ majorizes $\x$.
\end{definition}

To prove the main result of this section, we need the following lemma, which can be immediately obtained by comparing the coefficient of $x^{|V(H)|-1}$ on two sides of \eqref{inducedlmp}.
\begin{lemma}\label{coefficientlemma}
Let $H$ be a subgraph of a graph $G$.
Then the sum of zeros of $\beta(G, x)_{[H]}$ is equal to $\sum_{v\in V(H)}d_G(v)$.
In particular, the sum of zeros of $\beta(G, x)$ equals $2|E(G)|$.
\end{lemma}

\begin{theorem}\label{zerosmajorizedegree}
The zero sequence of the Laplacian matching polynomial of a graph majorizes its degree sequence.
\end{theorem}
\begin{proof}
 Let $d_G(v_1)\geq \cdots \geq d_G(v_n)$ be the degree sequence of a graph $G$,
and let $H_i$ be the subgraph of $G$ induced by the vertex set $\{v_1, \ldots,v_i\}$ for $1\leq i\leq n$.
We claim that for any $1\leq i\leq n$,
the zero sequence of $\beta(G, x)_{[H_i]}$  majorizes the degree sequence
$d_G(v_1)\geq  \cdots \geq d_G(v_{i})$, and the assertion follows form the case $i=n$.

We prove the claim by induction on $i$.
If $i=1$, then $\beta(G, x)_{[H_1]}=x-d_G(v_1)$. So the claim holds in this case.
Assume that the claim holds for $1< i<n$.
We now consider the case $i+1$.
By Theorem \ref{interlace}, the zeros of $\beta(G,x)_{[H_i]}$ interlace the zeros of $\beta(G,x)_{[H_{i+1}]}$.
Thus
the sequence of  the first largest $i$ zeros of $\beta(G,x)_{[H_{i+1}]}$
weakly majorizes
the sequence of the zeros of $\beta(G,x)_{[H_i]}$,
where the later majorizes the degree sequence $d_G(v_1)\geq  \cdots \geq d_G(v_{i})$ by the induction hypothesis.
This suggests that the sequence of  the  first largest $i$ zeros of $\beta(G,x)_{[H_{i+1}]}$ weakly majorizes the degree sequence $d_G(v_1) \geq \cdots \geq d_G(v_{i})$. Combining the fact provided by  Lemma \ref{coefficientlemma} that the sum of zeros of $\beta(G, x)_{[H_{i+1}]}$ is equal to
$\sum_{v\in V(H_{i+1})}d_G(v)$, the claim follows.
\end{proof}

Let $G$ be a connected graph of order $n\geq 2$ with degree sequence $d_1\geq  \cdots \geq d_n$.
The known result due to Grone \cite{GroneLMA} states that the sequence  of  the Laplacian eigenvalues  of $G$
majorizes the sequence $$(d_1+1,  d_2,\ldots, d_{n-1},d_n-1):=\mathbf{\widehat{d}}(G).$$
This raises a natural question:
Whether the zero sequence of $\beta(G, x)$ majorizes the sequence  $\mathbf{\widehat{d}}(G)$?
However, the answer is negative for all connected graphs since we have the following result.

\begin{proposition}\label{maxmimbound}
Let $G$ be a connected graph with minimum degree $1$.
Then the zero sequence of $\beta(G, x)$ majorizes the sequence  $\mathbf{\widehat{d}}(G)$ if and only if $G$ is a tree.
\end{proposition}
\begin{proof}
If the zero sequence of $\beta(G, x)$ majorizes the sequence  $\mathbf{\widehat{d}}(G)$,
then it is clear that $\lambda_{\min}(\beta(G, x))\leq  \delta(G)-1=0.$  Note that all zeros of the Laplacian matching polynomial of a graph are nonnegative. Then,  $\lambda_{\min}(\beta(G, x))=0$.
By \cite[Corrollay 2.3 (2)]{Zhang}, $G$ is a tree.

Conversely, if $G$ is a tree, then \cite[Theorem 3.3]{Mohammadian} suggests that $\beta(G, x)$ coincides with the Laplacian characteristic polynomial of $G$.
Thus, the zero sequence of $\beta(G, x)$ majorizes the sequence  $\mathbf{\widehat{d}}(G)$ by
Grone's theorem \cite{GroneLMA}.
\end{proof}

\section{A dual relation}\label{dualitysection}

This section is devoted  to establishing the dual relation between the Laplacian matching polynomial and the characteristic ploynomial of the signless Laplacian matrix.
Recall that the matrix $Q(G)=D(G)+A(G)$ is called the {\it signless Laplacian matrix}  of a graph $G$, where $D(G)$ denotes the diagonal matrix indexed by ${V}(G)$ with $d_G(v)$ in the $v$-th diagonal position.
Let $B_G$ be the vertex-edge incidence matrix of $G$.
Then one may write the adjacency matrix of $S_G$ as
\[
A(S_G) =
\left(
\begin{array}{cccc}
0& B_G\\
B_G^\top& 0\\
\end{array}
\right).
\]
It is well known that (see e.g., \cite[p. 63]{Cvet2}, \cite{Cvet3} or \cite[Theorem 1]{Zhou}):
\begin{equation} \label{add}
\phi(A(S_G), x)=x^{|E(G)|-|V(G)|} \phi(Q(G), x^2),
\end{equation}
which establishes an identity between $\phi(A(S_G), x)$ and $\phi(Q(G), x^2)$.
The following result is an extension of \eqref{add}, which can be considered as an analogue
of Theorem \ref{gsubdivisiontheorem}.

\begin{lemma}\label{subdivisionq1}
Let $G$ be a graph, and let $W$ be a subset of $V(G)$. Then
$$\phi(A(S_G-W), x)=x^{|E(G)|-|V(G)|+|W|} \phi(Q(G)_{[G-W]}, x^2),$$
where $Q(G)_{[G-W]}$ denotes the principle submatrix of $Q(G)$ indexed by $V(G-W)$.
\end{lemma}
\begin{proof}
Let $M$ be the submatrix of $B_G$ whose rows are indexed by $V(G)\setminus W$ and columns are indexed by $E(G)$.
It is not hard to check that $MM^\top=Q(G)_{[G-W]}$ and
\[
A(S_G-W) =
\left(
\begin{array}{cccc}
0& M\\
M^\top& 0\\
\end{array}
\right).
\]
Thus, using the Schur complement formula \cite[Theorem 2.7.1]{Brouwer}, we obtain that
\begin{eqnarray*}
\phi(A(S_G-W), x) &=&
 \det
\left(
\begin{array}{cccc}
xI_n& -M\\
-M^\top& xI_m\\
\end{array}
\right) \\
  ~ &=& x^{m-n}\det(x^2I_n-MM^\top)\\
  ~ &=& x^{|E(G)|-|V(G)|+|W|} \phi(Q(G)_{[G-W]}, x^2),
\end{eqnarray*}
where $n=|V(G)\setminus W|$ and $m=|E(G)|$, as desired.
\end{proof}

We are now ready to state and prove the main result of this section, which can be viewed as an analogue of Theorem \ref{matchingdualitythm}.

\begin{theorem}\label{Qdualitythm2}
Let $\mathcal{C}(G)$ be the set of all $2$-regular subgraphs of $G$. Then
\begin{equation}\label{Qdualitythm2eq2}
\phi(Q(G), x)= \beta(G,x)+ \sum_{C \in \mathcal{C}(G)}(-2)^{\omega(C)}\beta(G, x)_{[G-C]},
\end{equation}
and
\begin{equation}\label{Qdualitythm2eq1}
\beta(G,x)=\phi(Q(G), x)+\sum_{C \in \mathcal{C}(G)}2^{\omega(C)} \phi(Q(G)_{[G-C]}, x).
\end{equation}
\end{theorem}

\begin{proof}
	We begin with the proof of \eqref{Qdualitythm2eq2}.
Applying \eqref{matchingdualityeq1} to the graph $S_G$, we have
	\begin{align}\label{Dualpfequality211}
	\phi(A(S_G), x)
		=\alpha(S_G,x)+\sum_{C \in \mathcal{C}(S_G)}(-2)^{\omega(C)}\alpha(S_G-C, x).
	\end{align}
	Note that $\mathcal{C}(S_G)=\{S_C:C\in \mathcal{C}(G)\}$. Substituting \eqref{add} and \eqref{subdivisiontheoremeqqq} into \eqref{Dualpfequality211},  we can deduce that 
	\begin{equation}\label{Dualpfequality214}
		\phi(Q(G), x^2)=\beta(G,x^2)
		+x^{|V(G)|-|E(G)|}\sum_{C \in \mathcal{C}(G)}(-2)^{\omega(C)}\alpha(S_G-S_C, x).
	\end{equation}
	Moreover, for any $C \in \mathcal{C}(G)$, obverse that $|V(C)|=|E(C)|$
	and $S_G-V(C)$  consists of $S_G-S_C$ and $|E(C)|$ isolated vertices.
	So one may check that
	\begin{equation} \label{Dualpfequality216}
		\alpha(S_G-S_C, x)
		=\alpha(S_G-V(C), x) x^{-|E(C)|}
		=x^{|E(G)|-|V(G)|} \beta(G, x^2)_{[G-C]},
	\end{equation}
	where the last equality follows from Theorem \ref{gsubdivisiontheorem}.
	Combining \eqref{Dualpfequality214} and \eqref{Dualpfequality216}, we obtain that
	$$
	\phi(Q(G), x^2)=\beta(G,x^2)
	+\sum_{C \in \mathcal{C}(S_G)}(-2)^{\omega(C)}\beta(G, x^2)_{[G-C]},
	$$
which implies \eqref{Qdualitythm2eq2}, as desired.

	It remains to prove \eqref{Qdualitythm2eq1}.
	Applying \eqref{matchingdualityeq2} to the graph $S_G$, we have
	\begin{equation} \label{add8}
	\alpha(S_G,x)=\phi(A(S_G), x)+\sum_{C \in \mathcal{C}(S_G)}2^{\omega(C)}\phi(A(S_G-C), x).
	\end{equation}
Substituting  \eqref{subdivisiontheoremeqqq} and \eqref{add} into \eqref{add8}, we can deduce that 
\begin{equation}\label{Dualpfequality224}
\beta(G, x^2)=\phi(Q(G), x^2)
+x^{|V(G)|-|E(G)|}\sum_{C \in \mathcal{C}(G)}2^{\omega(C)}\phi(A(S_G-S_C), x).
\end{equation}
By a discussion similar to \eqref{Dualpfequality216}, we have
	\begin{equation} \label{Dualpfequality226}
\phi (A(S_G-S_C), x )
=\phi (A(S_G-V(C)), x )x^{-|E(C)|}
=x^{|E(G)|-|V(G)|}   \phi(Q(G)_{[G-C]}, x^2),
\end{equation}
where the last equality follows from Lemma \ref{subdivisionq1}. Combining \eqref{Dualpfequality224} and \eqref{Dualpfequality226}, we obtain that 
	$$
	\beta(G,x^2)=\phi(Q(G), x^2)+\sum_{C \in \mathcal{C}(G)}2^{\omega(C)} \phi(Q(G)_{[G-C]}, x^2),
	$$
	which implies \eqref{Qdualitythm2eq1}. 	The proof is completed.
\end{proof}

By Theorem \ref{matchingdualitythm}, $G$ is a forest if and only if $\alpha(G, x)=\phi(A(G), x)$, see \cite[Theorem 4.3]{Cvet} or \cite[Corollary 4.2]{Godsil1}.
Mohammadian \cite[Theorem 3.3]{Mohammadian} prove that $G$ is a forest if and only if  $\beta(G, x)=\phi(L(G), x)$. The following analogue of these facts is immediately obtained by  Theorem \ref{Qdualitythm2}.

\begin{corollary}\label{forest}
A graph $G$ is a forest if and only if $\beta(G, x)=\phi(Q(G), x)$.
\end{corollary}

In the following consequence, we give an upper bounds on the largest zero of
the Laplacian matching polynomial.

\begin{corollary}\label{mQinequality}
Let $G$ be a connected graph. Then
$$
 \lambda_{\max}(\beta(G, x))\leq \rho(Q(G)),
$$
with equality holds if and only if $G$ is a tree, where $\rho(Q(G))$ is the spectral radius of $Q(G)$.
\end{corollary}

\begin{proof}
Note that $Q(G)_{[G-C]}$ is a proper principle submatrix of $Q(G)$ for each $C \in \mathcal{C}(G)$.
Applying the Perron--Frobenius theorem \cite[Theorem 2.2.1]{Brouwer}, we have
$$\rho( Q(G)_{[G-C]} ) < \rho( Q(G)):=x_0.$$
So $\phi(Q(G)_{[G-C]}, x) >0$ whenever $x\geq x_0$.
It follows from Theorem \ref{Qdualitythm2} that
$$\beta(G,x)=\phi(Q(G), x)+\sum_{C \in \mathcal{C}(G)}2^{\omega(C)}\phi(Q(G)_{[G-C]}, x)\geq0$$
whenever $x\geq x_0$, which implies that  $\lambda_{\max}(\beta(G, x))\leq x_0$. Clearly, the equality holds if and only if $\mathcal{C}(G)=\emptyset$, that is, $G$ is a tree. This completes the proof.
\end{proof}

Corollary \ref{mQinequality} implies that every upper bound of the spectral radius of $Q(G)$ also holds for the largest root of $\beta(G, x)$.
For example, it is known that \cite[Proposition 3.9.1]{Brouwer} if $G$ has at least one edge, then
$$\lambda_{\max}( Q(G) )\leq \max_{\{u,v\}\in E(G)}(d_G(u)+d_G(v)),$$
with equality if and only if $G$ is regular or bipartite semiregular.
By combining this fact and Corollary \ref{mQinequality}, we immediately obtain that
\begin{equation*}\label{uppereq}
 \lambda_{\max}(\beta(G, x)) \leq \max_{\{u,v\}\in E(G)}\big(d_G(u)+d_G(v)\big),
\end{equation*}
with equality holds if and only if $G$ is a star.

\section{Concluding Remarks}

As shown in the present paper, the subdivision method provides a bridge between the matching polynomial and the Laplacian matching polynomial.
Using the method, one may transform the problem of the Laplacian matching polynomial of a graph $G$
into the problem of the matching polynomial of $S_G$, and vice versa.
It would be interesting to use this method to find more properties of the Laplacian matching polynomial.

Recently, Li and Yan \cite{LiYan} generalized the definition of the Laplacian matching polynomial from simple graphs to edge-weighted graphs as follows.
Let $G$ be a graph with the edge weight function $\w:E(G)\rightarrow \R^+$.
The {\it Laplacian matching polynomial} of the edge-weighted graph $(G,\w)$ is defined as
$$
 \beta(G,\w, x)=\sum_{M\in\mathcal{M}(G)} (-1)^{|M|} \left(  \prod_{e \in E(M)}\w(e)^2\right)
\left(\prod_{v\in V(G)\setminus V(M)}\big(x-\w(v)\big)\right),
$$
where $\w(v)$ is the sum of weights of edges incident with vertex $v$ in $G$.
Observe that $\beta(G,\w, x)=\beta(G, x)$ when $\w(e)=1$ for each $e \in E(G)$.
Similar to \eqref{subdivisiontheoremeqqq}, Li and Yan \cite[Theorem 3.3]{LiYan} proved an identity between $\beta(G,\w, x)$ and the matching polynomial of the edge-weighted subdivision graph of $(G,\w)$.
The authors would believe that the subdivision method  can be used to obtain some properties of the Laplacian matching polynomial of edge-weighted graphs, possibly including some appropriate adjustments.

\vspace{5mm}
\noindent{{\bf Acknowledgements}}
The authors would like to thank Ali Mohammadian for his helpful discussions and comments.

\noindent{{\bf Data availability statement}}  This manuscript has no associated data.

\noindent{{\bf Conflicts of interest statement}} The authors declare no conflict of interest to the content of this article.


\end{document}